\documentclass[12pt,thmsa]{article}
\usepackage{amsmath,amssymb,amsthm}

\usepackage{graphics,wrapfig}
\usepackage{latexsym}

\newtheorem{theorem}{Theorem}

\newtheorem{conjecture}{Conjecture}

\newtheorem{corollary}{Corollary}
\newtheorem{lemma}{Lemma}

\begin{document}

\author{Peng-An Chen \\
Department of Applied Mathematics\\
National Taitung University\\
Taitung, Taiwan}
\date{Nov 17, 2017}
\title{On the chromatic number of almost $s$-stable Kneser graphs}
\maketitle

\begin{abstract}
In 2011, Meunier conjectured that for positive integers $n,k,r,s$ with $%
k\geq 2$, $r\geq 2$, and $n\geq \max (\{r,s\})k$, the chromatic number of $s$%
-stable $r$-uniform Kneser hypergraphs is equal to $\left\lceil \frac{n-\max
(\{r,s\})(k-1)}{r-1}\right\rceil $. It is a strengthened version of the
conjecture proposed by Ziegler (2002), and Alon, Drewnowski and \L uczak
(2009). The problem about the chromatic number of almost $s$-stable $r$%
-uniform Kneser hypergraphs has also been introduced by Meunier (2011).

For the $r=2$ case of the Meunier conjecture, Jonsson (2012) provided a
purely combinatorial proof to confirm the conjecture for $s\geq 4$ and $n$
sufficiently large, and by Chen (2015) for even $s$ and any $n$. The case $%
s=3$ is completely open, even the chromatic number of the usual almost $s$%
-stable Kneser graphs.

In this paper, we obtain a topological lower bound for the chromatic number
of almost $s$-stable $r$-uniform Kneser hypergraphs via a different
approach. For the case $r=2$, we conclude that the chromatic number of
almost $s$-stable Kneser graphs is equal to $n-s(k-1)$ for all $s\geq 2$.
Set $t=n-s(k-1)$. We show that any proper coloring of an almost $s$-stable
Kneser graph must contain a completely multicolored complete bipartite
subgraph $K_{\left\lceil \frac{t}{2}\right\rceil \left\lfloor \frac{t}{2}%
\right\rfloor }$. It follows that the local chromatic number of almost $s$%
-stable Kneser graphs is at least $\left\lceil \frac{t}{2}\right\rceil +1$.
It is a strengthened result of Simonyi and Tardos (2007), and Meunier's
(2014) lower bound for almost $s$-stable Kneser graphs.
\end{abstract}

\section{\textbf{Introduction}}

Let $[n]$ denote the set $\left\{ 1,2,\ldots ,n\right\} $. For every
nonempty subset $S$ of $[n]$, $\max \left( S\right) $ denotes the maximal
element of $S$. In particular, we define $\max ($\O $)$\ = $0$. A mapping $%
c:V\longrightarrow \lbrack m]$ is a \textit{proper} \textit{coloring} of a
graph $G=\left( V,E\right) $ with $m$ colors if none of the edges $e\in E$
is monochromatic under $c$. The \textit{chromatic number} $\chi (G)$ of a
graph $G$ is the smallest number $m$ such that a proper coloring $%
c:V\longrightarrow \lbrack m]$ exists. For positive integers $n,k$ and $s$,
a $k$-subset $S\subseteq $ $[n]$ is $s$\textit{-stable} ( resp. \textit{%
almost} $s$\textit{-stable}) if $\left\vert S\right\vert =k$ and any two of
its elements are at least "at distance $s$ apart" on the $n$-cycle ( resp. $%
n $-path), that is, if $s\leq \left\vert i-j\right\vert \leq n-s$ ( resp. $%
\left\vert i-j\right\vert \geq s$) for distinct $i,j\in S$. Hereafter, the
symbols $\binom{\left[ n\right] }{k}$, $\binom{\left[ n\right] }{k}_{s\text{-%
}stab}$, and $\binom{\left[ n\right] }{k}_{\widetilde{s\text{-}stab}}$ stand
for the collection of all $k$-subsets of $[n]$, the collection of all $s$%
-stable $k$-subsets of $[n]$, and the collection of all almost $s$-stable $k$%
-subsets of $[n]$, respectively. Choosing $k=1$, we obtain that $\binom{%
\left[ n\right] }{1}=\binom{\left[ n\right] }{1}_{s\text{-}stab}=\binom{%
\left[ n\right] }{1}_{\widetilde{s\text{-}stab}}$. One can easily check that 
$\binom{\left[ n\right] }{k}_{s\text{-}stab}\subseteq \binom{\left[ n\right] 
}{k}_{\widetilde{s\text{-}stab}}\subseteq \binom{\left[ n\right] }{k}$.

The \textit{Kneser graph}, denoted as $KG(n,k)$, is defined for positive
integers $n\geq 2k$ as the graph having $\binom{\left[ n\right] }{k}$ as
vertex set. Two vertices are defined to be adjacent in $KG(n,k)$ if they are
disjoint. Choosing $k=1$, we obtain the \textit{complete graph\ }$K_{n}$.
The $s$\textit{-stable Kneser graph}, denoted as $KG^{2}(n,k)_{s\text{-}%
stab} $, is defined for positive integers $n\geq sk$ as the graph having $%
\binom{\left[ n\right] }{k}_{s\text{-}stab}$ as vertex set. Two vertices are
defined to be adjacent in $KG^{2}(n,k)_{s\text{-}stab}$ if they are
disjoint. Choosing $s=1$, we obtain the Kneser graph\textit{\ }$KG(n,k)$,
whereas $s=2$ yields the \textit{stable Kneser graph }or \textit{Schrijver
graph}, denoted as $SG(n,k)$. The \textit{almost} $s$\textit{-stable Kneser
graph}, denoted as $KG^{2}(n,k)_{\widetilde{s\text{-}stab}}$, is defined for
positive integers $n\geq sk$ as the graph having $\binom{\left[ n\right] }{k}%
_{\widetilde{s\text{-}stab}}$ as vertex set. Two vertices are defined to be
adjacent in $KG^{2}(n,k)_{\widetilde{s\text{-}stab}}$ if they are disjoint.

Kneser \cite{Kneser55} conjectured that the chromatic number $\chi \left(
KG(n,k)\right) $ of the Kneser graph $KG(n,k)$ is equal to $n-2k+2$.
Kneser's conjecture \cite{Kneser55} was proved by Lov\'{a}sz \cite{Lovasz78}
using the Borsuk-Ulam theorem; all subsequent proofs, extensions and
generalizations also relied on Algebraic Topology results, namely the
Borsuk-Ulam theorem and its extensions. Matou\v{s}ek \cite{Matousek04}
provided the first combinatorial proof of Kneser's conjecture \cite{Kneser55}%
. Schrijver \cite{Schrijver78} found a fascinating family of subgraphs $%
SG\left( n,k\right) $ of $KG(n,k)$ that are vertex-critical with respect to
the chromatic number. Ziegler \cite{Ziegler02,Ziegler06} provided a
combinatorial proof of Schrijver's theorem \cite{Schrijver78}. Meunier \cite%
{Meunier11} provided another simple combinatorial proof of Schrijver's
theorem \cite{Schrijver78}.

A hypergraph $\mathcal{H}$ is a pair $\mathcal{H}=(V(\mathcal{H}),E(\mathcal{%
H}))$, where $V(\mathcal{H})$ is a finite set and $E(\mathcal{H})$ a family
of subsets of $V(\mathcal{H})$. The set $V(\mathcal{H})$ is called the
vertex set and the set $E(\mathcal{H})$ is called the edge set. Let $r$ be
any positive integer with $r\geq 2$. A hypergraph is said to be $r$-\textit{%
uniform} if all its edges $S$ $\in \mathcal{H}$ have the same cardinality $r$%
. A \textit{proper coloring} of a hypergraph $\mathcal{H}$ with $t$ colors
is a function $c:V\longrightarrow \lbrack t]$ so that no edge $S$ $\in 
\mathcal{H}$ is monochromatic, that is, every edge contains two elements $%
i,j\in S$ with $c(i)\neq c(j)$. Equivalently, no $c^{-1}(i)$ contains a set $%
S$ $\in \mathcal{H}$. The \textit{chromatic number} $\chi (\mathcal{H})$ of
a hypergraph $\mathcal{H}$ is the smallest number $t$ such that there exists
a $t$-coloring for $\mathcal{H}$. Throughout this paper, we suppose that $V(%
\mathcal{H})=[n]$ for some positive integer $n$.

For any hypergraph $\mathcal{H}=(V(\mathcal{H}),E(\mathcal{H}))$ and
positive integer $r\geq 2$, the \textit{general Kneser hypergraph} $KG^{r}(%
\mathcal{H})$ of $\mathcal{H}$ is an $r$-uniform hypergraph has $E(\mathcal{H%
})$ as its vertex set and the edge set consisting of all $r$-tuples of
pairwise disjoint edges of $\mathcal{H}$. The \textit{Kneser hypergraph} $%
KG^{r}\binom{\left[ n\right] }{k}$ is an $r$-uniform hypergraph which has $%
\binom{\left[ n\right] }{k}$ as vertex set and whose edges are formed by the 
$r$-tuples of disjoint $k$-element subsets of $[n]$. Choosing $r=2$, we
obtain the ordinary Kneser graph\textit{\ }$KG(n,k)$. Hereafter, for
positive integes $n,k,r$ with $r\geq 2$ and $n\geq rk$, the hypergraph $%
KG^{r}\binom{\left[ n\right] }{k}$ is denoted by $KG^{r}(n,k)$. Erd\H{o}s 
\cite{Erd76} conjectured that, for $n\geq rk$,%
\begin{equation*}
\chi \left( KG^{r}(n,k)\right) =\left\lceil \frac{n-r(k-1)}{r-1}\right\rceil 
\text{.}
\end{equation*}%
The conjecture settled by Alon, Frankl and Lov\'{a}sz \cite{AFL86}.
Above-mentioned results are generalized in many ways. One of the most
promising generalizations is the one found by Dol'nikov \cite{Dolnikov88}
and extended by K\v{r}\'{\i}\v{z} \cite{Kriz92,Kriz2000}. Finding a lower
bound for chromatic number of Kneser hypergraphs has been studied in the
literature, see \cite{Dolnikov88,Kriz92,Kriz2000,Sarkaria90,Simonyi and
Tardos06,Simonyi and Tardos07,Ziegler02,Ziegler06}.

The $r$-\textit{colorability defect} of $\mathcal{H}$, denoted by cd$^{r}(%
\mathcal{H})$, is the minimum number of vertices that should be removed from 
$\mathcal{H}$ so that the induced subhypergraphs on the remaining vertices
has the chromatic number at most $r$. Dol'nikov \cite{Dolnikov88} (for $r=2$%
) and K\v{r}\'{\i}\v{z} \cite{Kriz92,Kriz2000} proved that

\begin{theorem}[ Dol'nikov-K\v{r}\'{\i}\v{z} theorem]
\label{t00006}Let $\mathcal{H}$ be a hypergraph and assume that \O\ is not
an edge of $\mathcal{H}$. Then 
\begin{equation*}
\chi \left( KG^{r}(\mathcal{H})\right) \geq \left\lceil \frac{\text{cd}^{r}(%
\mathcal{H})}{r-1}\right\rceil
\end{equation*}%
for any integer $r\geq 2$.
\end{theorem}

These results were also generalized by Ziegler \cite{Ziegler02,Ziegler06}.
Note that if we set $\mathcal{H}=\left( [n],\binom{\left[ n\right] }{k}%
\right) $, then this result implies $\chi \left( KG^{r}(n,k)\right)
=\left\lceil \frac{n-r(k-1)}{r-1}\right\rceil $ bringing in Alon, Frankl and
Lov\'{a}sz's result \cite{AFL86}.

The $s$-\textit{stable }$r$-\textit{uniform} \textit{Kneser hypergraph} $%
KG^{r}\binom{\left[ n\right] }{k}_{s\text{-}stable}$ is an $r$-uniform
hypergraph which has $\binom{\left[ n\right] }{k}_{s\text{-}stable}$ as
vertex set and whose edges are formed by the $r$-tuples of disjoint $s$%
-stable $k$-element subsets of $[n]$. Note that an $s$-stable $2$-uniform\
Kneser hypergraph is simply an $s$\textit{-stable Kneser graph}. Hereafter,
for positive integes $n,k,r,s$ with $k\geq 2$, $r\geq 2$ and $n\geq \max
(\{r,s\})k$, the hypergraph $KG^{r}\binom{\left[ n\right] }{k}_{s\text{-}%
stable}$ is denoted by $KG^{r}(n,k)_{s\text{-}stable}$.

The \textit{almost} $s$-\textit{stable }$r$\textit{-uniform\ Kneser
hypergraph} $KG^{r}\binom{\left[ n\right] }{k}_{\widetilde{s\text{-}stab}}$
is an $r$-uniform hypergraph which has $\binom{\left[ n\right] }{k}_{%
\widetilde{s\text{-}stab}}$ as vertex set and whose edges are formed by the $%
r$-tuples of disjoint almost $s$-stable $k$-element subsets of $[n]$. Note
that an almost $s$-stable $2$-uniform\ Kneser hypergraph is simply an 
\textit{almost }$s$\textit{-stable Kneser graph}. Hereafter, for positive
integes $n,k,r,s$ with $k\geq 2$, $r\geq 2$ and $n\geq \max (\{r,s\})k$, the
hypergraph $KG^{r}\binom{\left[ n\right] }{k}_{\widetilde{s\text{-}stab}}$
is denoted by $KG^{r}(n,k)_{\widetilde{s\text{-}stab}}$.

Ziegler \cite{Ziegler02,Ziegler06} gave a combinatorial proof of the
Alon-Frankl-Lov\'{a}sz theorem \cite{AFL86}. He was inspired by a
combinatorial proof of the Lov\'{a}sz theorem found by Matou\v{s}ek \cite%
{Matousek04}. At the end of his paper, Ziegler made the supposition that $%
\chi \left( KG^{r}(n,k)_{r\text{-}stab}\right) =\chi \left(
KG^{r}(n,k)\right) $ for any $n\geq rk$. Alon, Drewnowski and \L uczak make
this supposition an explicit conjecture in \cite{ADL09}.

\begin{conjecture}
\label{c000}Let $n,k$,and $r$ be positive integers such that $k\geq 2$, $%
r\geq 2$, and $n\geq rk$. Then $\chi \left( KG^{r}(n,k)_{r\text{-}%
stab}\right) =\left\lceil \frac{n-r(k-1)}{r-1}\right\rceil =\chi \left(
KG^{r}(n,k)\right) $.
\end{conjecture}

Alon, Drewnowski, and \L uczak \cite{ADL09} also confirmed Conjecture \ref%
{c000} for $r$ is a power of $2$. Meunier \cite{Meunier11} proposed the
following conjecture:

\begin{conjecture}
\label{c0001}Let $n,k,r$,and $s$ be positive integers such that $k\geq 2$, $%
r\geq 2$, and $n\geq \max (\{r,s\})k$. Then $\chi \left( KG^{r}(n,k)_{s\text{%
-}stab}\right) =\left\lceil \frac{n-\max (\{r,s\})(k-1)}{r-1}\right\rceil $.
\end{conjecture}

The case $s=r$ is the Alon-Drewnowski-\L uczak-Ziegler conjecture \cite%
{ADL09,Ziegler02,Ziegler06}. Conjecture \ref{c0001} is a strengthened
version of Conjecture \ref{c000}.

Ziegler \cite{Ziegler02,Ziegler06} pointed out that for positive integes $%
n,k,r,s$ with $r\geq 2$, $k\geq 2$, and $n\geq sk$, cd$^{r}\binom{\left[ n%
\right] }{k}_{s\text{-}stab}=\max (\{n-rs(k-1),0\})$. Since $KG^{r}(n,k)_{s%
\text{-}stab}$ is an induced subhypergraph of $KG^{r}(n,k)_{\widetilde{s%
\text{-}stab}}$, cd$^{r}\binom{\left[ n\right] }{k}_{\widetilde{s\text{-}stab%
}}\geq $ cd$^{r}\binom{\left[ n\right] }{k}_{s\text{-}stab}$. We can easily
obtain the following $r$-colorability defect result of $\binom{\left[ n%
\right] }{k}_{\widetilde{s\text{-}stab}}$ with the same proof work of
Ziegler \cite{Ziegler02,Ziegler06}.

\begin{lemma}
\label{l10} Let $n,k,r$,and $s$ be positive integers such that $r\geq 2$, $%
k\geq 2$, and $n\geq sk$. Then cd$^{r}\binom{\left[ n\right] }{k}_{%
\widetilde{s\text{-}stab}}=$ cd$^{r}\binom{\left[ n\right] }{k}_{s\text{-}%
stab}=\max (\{n-rs(k-1),0\})$.
\end{lemma}

Combining the Dol'nikov-K\v{r}\'{\i}\v{z} Theorem \ref{t00006} and Lemma \ref%
{l10}, we have that%
\begin{equation*}
\chi \left( KG^{r}(n,k)_{s\text{-}stab}\right) \geq \left\lceil \frac{%
n-rs(k-1)}{r-1}\right\rceil
\end{equation*}%
and%
\begin{equation*}
\chi \left( KG^{r}(n,k)_{\widetilde{s\text{-}stab}}\right) \geq \left\lceil 
\frac{n-rs(k-1)}{r-1}\right\rceil
\end{equation*}%
for any $n\geq rs(k-1)+\max \left( \{r,s\}\right) $. We would like to find
the substantial improvement of the Dol'nikov-K\v{r}\'{\i}\v{z} lower bounds
for $\chi \left( KG^{r}(n,k)_{s\text{-}stab}\right) $ and $\chi \left(
KG^{r}(n,k)_{\widetilde{s\text{-}stab}}\right) $.

First, we prove the following lemma, the proof follows a very similar scheme
as the proofs of Ziegler \cite{Ziegler02,Ziegler06} and Meunier \cite%
{Meunier11}.

\begin{lemma}
\label{l00001}Let $n,k,r$, and $s$ be positive integers such that $k\geq 2$, 
$r\geq 2$, and $n\geq \max (\{r,s\})k$. Then $\chi \left( KG^{r}(n,k)_{s%
\text{-}stab}\right) \leq \chi \left( KG^{r}(n,k)_{\widetilde{s\text{-}stab}%
}\right) \leq \left\lceil \frac{n-\max (\{r,s\})(k-1)}{r-1}\right\rceil $.
\end{lemma}

\begin{proof}
Since $KG^{r}(n,k)_{\widetilde{s\text{-}stab}}$ is an induced subhypergraph
of $KG^{r}(n,k)$. If $r>s$, then $\chi \left( KG^{r}(n,k)_{s\text{-}%
stab}\right) \leq \chi \left( KG^{r}(n,k)_{\widetilde{s\text{-}stab}}\right)
\leq \chi \left( KG^{r}(n,k)\right) =\left\lceil \frac{n-r(k-1)}{r-1}%
\right\rceil =\left\lceil \frac{n-\max (\{r,s\})(k-1)}{r-1}\right\rceil $.

Assume that $r\leq s$. Let $t=\left\lceil \frac{n-s(k-1)}{r-1}\right\rceil $%
. Define the map 
\begin{equation*}
c:\binom{\left[ n\right] }{k}_{\widetilde{s\text{-}stab}}\longrightarrow
\lbrack t]
\end{equation*}%
as follows:

Assume $S=\{i_{1},i_{2},\ldots ,i_{k}\}\in \binom{\left[ n\right] }{k}_{%
\widetilde{s\text{-}stab}}$ with $1\leq i_{1}<i_{2}<\cdots <i_{k}\leq n$.
Since $S$ is almost $s$-stable, $i_{k-j}\leq n-sj$ for $j=0,\ldots ,k-1$.
Then we define its color%
\begin{equation*}
c(S):=\left\lceil \frac{i_{1}}{r-1}\right\rceil =\left\lceil \frac{\min (S)}{%
r-1}\right\rceil \text{.}
\end{equation*}%
Thus we obtain a value $c(S)$ in $[t]$. One can easily check that $c$ is a
proper coloring of $KG^{r}\left( n,k\right) _{\widetilde{s\text{-}stab}}$.
The proof is complete.
\end{proof}

It is obvious that the $s$-stable $r$-uniform\ Kneser hypergraph $%
KG^{r}(n,k)_{s\text{-}stab}$ is an induced subhypergraph of the almost $s$%
-stable $r$-uniform\ Kneser hypergraph $KG^{r}\left( n,k\right) _{\widetilde{%
s\text{-}stab}}$. Combining Lemma \ref{l00001}, Conjecture \ref{c000}, and
Conjecture \ref{c0001}, we propose two conjectures with related to the
almost $s$-stable $r$-uniform\ Kneser hypergraph $KG^{r}\left( n,k\right) _{%
\widetilde{s\text{-}stab}}$ as follows.

\begin{conjecture}
\label{c0002}Let $n,k$, and $r$ be positive integers such that $k\geq 2$, $%
r\geq 2$, and $n\geq rk$. Then $\chi \left( KG^{r}(n,k)_{\widetilde{r\text{-}%
stab}}\right) =\left\lceil \frac{n-r(k-1)}{r-1}\right\rceil =\chi \left(
KG^{r}(n,k)\right) $.
\end{conjecture}

\begin{conjecture}
\label{c0004}Let $n,k,r$, and $s$ be positive integers such that $k\geq 2$, $%
r\geq 2$, and $n\geq \max (\{r,s\})k$. Then $\chi \left( KG^{r}(n,k)_{%
\widetilde{s\text{-}stab}}\right) =\left\lceil \frac{n-\max (\{r,s\})(k-1)}{%
r-1}\right\rceil $.
\end{conjecture}

\section{Main results}

In this section, we first discuss some proper coloring results about the
chromatic number of $KG^{r}(n,k)_{\widetilde{s\text{-}stab}}$ for $r\leq s$.
In addition, we show that Conjecture \ref{c0002} is equivalent to Conjecture %
\ref{c0004}. Next, we obtain a topological lower bound of $\chi \left(
KG^{p}(n,k)_{\widetilde{s\text{-}stab}}\right) $ for prime $p$. Finally, we
completely confirm the chromatic number of almost $s$-stable Kneser graphs
is equal to $n-s(k-1)$ for all $s\geq 2$. Let $t=n-s(k-1)$. We also show
that for any proper coloring of $KG^{2}\left( n,k\right) _{\widetilde{s\text{%
-}stab}}$ with $\{1,2,\ldots ,m\}$ colors ($m$ arbitrary) must contain a
completely multicolored complete bipartite subgraph $K_{\left\lceil \frac{t}{%
2}\right\rceil \left\lfloor \frac{t}{2}\right\rfloor }$. Moreover, we obtain
that the local chromatic number of the usual almost $s$-stable Kneser graphs
is at least $\left\lceil \frac{t}{2}\right\rceil +1$.

\subsection{$\protect\chi \left( KG^{r}(n,k)_{\widetilde{s\text{-}stab}%
}\right) $ for $r\leq s$}

In this subsection, we would like to investigate $\chi \left( KG^{r}(n,k)_{%
\widetilde{s\text{-}stab}}\right) $ for $r\leq s$. First, we provide the
following lemma.

\begin{lemma}
\label{l12}Let $n$, $k$, $r$, and $s$ be positive integers such that $k\geq
2 $, $r\geq 2$, $s\geq r_{1}\geq 2$, and $n\geq sk$. Assume that $\chi
\left( KG^{r_{1}}(n,k)_{\widetilde{s\text{-}stab}}\right) =\left\lceil \frac{%
n-s(k-1)}{r_{1}-1}\right\rceil $. Then for any $\ r\leq r_{1}$, we have $%
\chi \left( KG^{r}(n,k)_{\widetilde{s\text{-}stab}}\right) =\left\lceil 
\frac{n-s(k-1)}{r-1}\right\rceil $.
\end{lemma}

\begin{proof}
Let $r$ be a positive integer such that $r_{1}>r\geq 2$. It suffices to show
that $\chi \left( KG^{r}(n,k)_{\widetilde{s\text{-}stab}}\right)
=\left\lceil \frac{n-s(k-1)}{r-1}\right\rceil $. Assume that $KG^{r}(n,k)_{%
\widetilde{s\text{-}stab}}$ is properly colored with $t$ colors. Set $%
\overline{n}=n+(r_{1}-r)t\geq n\geq sk$. There exists a proper $t$-coloring $%
c:\binom{\left[ n\right] }{k}_{\widetilde{s\text{-}stab}}\longrightarrow %
\left[ t\right] $ of the almost $s$-stable $r$-uniform Kneser hypergraph $%
KG^{r}(n,k)_{\widetilde{s\text{-}stab}}$. Now we construct a proper $t$%
-coloring $f$ of the almost $s$-stable $r_{1}$-uniform Kneser hypergraph $%
KG^{s}(\overline{n},k)_{\widetilde{s\text{-}stab}}$. \newline
Let the map 
\begin{equation*}
f:\binom{\left[ \overline{n}\right] }{k}_{\widetilde{s\text{-}stab}%
}\longrightarrow \left[ t\right]
\end{equation*}%
be defined as follows:

Assume $S=\{i_{1},i_{2},\ldots ,i_{k}\}\in \binom{\left[ \overline{n}\right] 
}{k}_{\widetilde{s\text{-}stab}}$ with $1\leq i_{1}<i_{2}<\cdots <i_{k}\leq 
\overline{n}$. Since $S$ is almost $s$-stable, $i_{k-j}\leq \overline{n}-sj$
for $j=0,\ldots ,k-1$.

\textbf{Case I.} $S\subseteq \lbrack n]$, that is $\max \left( S\right) \leq
n$. We know that if $S\in \binom{\left[ \overline{n}\right] }{k}_{\widetilde{%
s\text{-}stab}}$ and $S\subseteq \lbrack n]\subseteq \left[ \overline{n}%
\right] $, then $S\in \binom{\left[ n\right] }{k}_{\widetilde{s\text{-}stab}%
} $. So we set%
\begin{equation*}
f(S)=c(S)\text{.}
\end{equation*}%
Clearly, $f$ is well-defined. Hence we obtain a value $f(S)$ in $\left[ t%
\right] $.

\textbf{Case II.} $S\backslash \lbrack n]\neq $ \O , that is, $\max \left(
S\right) >n$. Then set%
\begin{equation*}
f(S)=\left\lceil \frac{\max \left( S\right) -n}{r_{1}-r}\right\rceil \text{.}
\end{equation*}%
One can easily check that $f(S)\leq \left\lceil \frac{\overline{n}-n}{r_{1}-r%
}\right\rceil =t$. Thus, we obtain a value $f(S)$ in $\left[ t\right] $.%
\newline

We claim that $f$ is a proper $t$-coloring of $KG^{r_{1}}(\overline{n},k)_{%
\widetilde{s\text{-}stab}}$. Suppose to contrary that there are $r_{1}$
pairwise disjoint $k$-subsets $T_{1},T_{2},\ldots ,T_{r_{1}}\in \binom{\left[
\overline{n}\right] }{k}_{\widetilde{s\text{-}stab}}$\ and $m\in \lbrack t]$
such that $f(T_{1})=f(T_{2})=\cdots =f(T_{r_{1}})=m$. Without lose of
generality let $\max \left( T_{1}\right) <\max \left( T_{2}\right) <\cdots
<\max \left( T_{r_{1}}\right) $. We have $\max \left( T_{r_{1}}\right)
>\cdots >\max \left( T_{r}\right) >n$, otherwise $T_{1},T_{2},\ldots
,T_{r}\in \binom{\left[ n\right] }{k}_{\widetilde{s\text{-}stab}}$ and $%
c(T_{1})=c(T_{2})=\cdots =c(T_{r})=m$ by the definition of $f$. It is
impossible since $c$ is a proper coloring of $KG^{r}(n,k)_{\widetilde{s\text{%
-}stab}}$. Let $q\in \lbrack r_{1}]$ such that%
\begin{equation*}
\max \left( T_{i}\right) \leq n\text{ for }1\leq i\leq q\text{,}
\end{equation*}%
and%
\begin{equation*}
\max \left( T_{i}\right) >n\text{ for }q+1\leq i\leq r_{1}\text{.}
\end{equation*}%
Hence $q<r$. We know that $T_{1},T_{2},\ldots ,T_{q}\in \binom{\left[ n%
\right] }{k}_{\widetilde{s\text{-}stab}}$. By definiton of $f$, we have $%
f\left( T_{i}\right) =c\left( T_{i}\right) =m$ for $1\leq i\leq q$, and $%
f\left( T_{i}\right) =\left\lceil \frac{\max \left( T_{i}\right) -n}{r_{1}-r}%
\right\rceil =m$ for $q+1\leq i\leq r_{1}$. Since $q<r$, and then $%
r_{1}-q>r_{1}-r$. It means that there are more than $r_{1}-r$ distinct
positive integers $\max \left( T_{q+1}\right) ,\max \left( T_{q+2}\right)
,\ldots ,\max \left( T_{r_{1}}\right) $ such that $\left\lceil \frac{\max
\left( T_{q+1}\right) -n}{r_{1}-r}\right\rceil =\left\lceil \frac{\max
\left( T_{q+2}\right) -n}{r_{1}-r}\right\rceil =\cdots =\left\lceil \frac{%
\max \left( T_{r_{1}}\right) -n}{r_{1}-r}\right\rceil $ $=m$. It is
impossible. So we are done.\newline

Hence, $t\geq \left\lceil \frac{\overline{n}-s(k-1)}{r_{1}-1}\right\rceil $.
It follows that $t\geq \frac{n+(r_{1}-r)t-s(k-1)}{r_{1}-1}%
\Longleftrightarrow (r_{1}-1)t\geq n+(r_{1}-r)t-s(k-1)\Leftrightarrow
(r-1)t\geq n-s(k-1)$. Hence, we have 
\begin{equation*}
t\geq \left\lceil \frac{n-s(k-1)}{r-1}\right\rceil \text{,}
\end{equation*}%
that is, $\chi \left( KG^{r}(n,k)_{\widetilde{s\text{-}stab}}\right) \geq
\left\lceil \frac{n-s(k-1)}{r-1}\right\rceil $. From Lemma \ref{l00001}, we
conclude that $\chi \left( KG^{r}(n,k)_{\widetilde{s\text{-}stab}}\right)
=\left\lceil \frac{n-s(k-1)}{r-1}\right\rceil $.
\end{proof}

Conjecture \ref{c0001} is confirmed by Jonsson \cite{Jonsson12} for $r=2^{p}$
and $s=2^{q}$ with positive integers $p\leq q$, and by Chen \cite{Pengan15}
for $r$ is a power of $2$ and $s$ is a multiple of $r$. Combining Chen's
work \cite{Pengan15}, Lemma \ref{l00001}, and Lemma \ref{l12}, we obtain the
following immediate consequence about Conjecture \ref{c0004}:

\begin{corollary}
Let $n,k$, $q$, $r$ and $\alpha $ be positive integers such that $k\geq 2$
and $r\geq 2$. Assume that $2^{q}\geq r$. Then $\chi \left( KG^{r}(n,k)_{%
\widetilde{\alpha 2^{q}\text{-}stab}}\right) =\left\lceil \frac{n-\alpha
2^{q}(k-1)}{r-1}\right\rceil $ for any $n\geq \alpha 2^{q}k$.
\end{corollary}

Moreover, we obtain an immediate consequence of Lemma \ref{l12}.

\begin{theorem}
\label{t07}Conjecture \ref{c0002} is equivalent to Conjecture \ref{c0004}.
\end{theorem}

So it remains an interesting issue to verify the logical equivalence between
Conjecture \ref{c000} and Conjecture \ref{c0001}.

\subsection{Topological lower bound of $\protect\chi \left( KG^{p}(n,k)_{%
\widetilde{s\text{-}stab}}\right) $ for prime $p$}

This subsection is devoted to find the topological lower bounds of the
chromatic number of $KG^{p}(n,k)_{\widetilde{s\text{-}stab}}$ and the
chromatic number of $KG^{p}(n,k)_{s\text{-}stab}$ for prime $p$. First, we
introduce the $Z_{p}$-Tucker lemma.

Throughout this paper, for any positive integer $p$, we assume that $Z_{p}$ $%
=\{\omega ,\omega ^{2},\ldots ,\omega ^{p}\}$ is the cyclic and
muliplicative group of the $p$th roots of unity. We emphasize that $0$ is
not considered as an element of $Z_{p}$. We write $(Z_{p}\cup \{0\})^{n}$
for the set of all \textit{signed subsets} of $[n]$. We define $\left\vert
X\right\vert $ to be the quantity $\left\vert \{i\in \lbrack n]:x_{i}\neq
0\}\right\vert $.

Any element $X=(x_{1},x_{2},\ldots ,x_{m})\in (Z_{p}\cup \{0\})^{n}$ can
alternatively and without further mention be denoted by a $p$-tuple $%
X=(X_{1},X_{2},\ldots ,X_{p})$ where $X_{j}:=\left\{ i\in \lbrack
n]:x_{i}=\omega ^{j}\right\} $. Note that the $X_{j}$ are then necessarily
disjoint. For two elements $X,Y\in (Z_{p}\cup \{0\})^{n}$, we denote by $%
X\subseteq Y$ the fact that for all $j\in \lbrack p]$ we have $%
X_{j}\subseteq Y_{j}$. When $X\subseteq Y$, note that the sequence of
non-zero terms in $(x_{1},x_{2},\ldots ,x_{n})$ is a subsequence of $%
(y_{1},y_{2},\ldots ,y_{n})$. For any $X\in (Z_{p}\cup \{0\})^{n}\backslash
\{(0,\ldots ,0)\}$, we write $\max \left( X\right) $ for the maximal element
of $X_{1}\cup X_{2}\cup \ldots \cup X_{p}$, that is, $\max \left( X\right)
=\max \left( X_{1}\cup X_{2}\cup \ldots \cup X_{p}\right) $.

Tucker's combinatorial lemma \cite{Tucker46} and Fan's combinatorial lemma 
\cite{Ky Fan52} are two powerful tools in combinatorial topology. The
problem of finding a lower bound for the chromatic number of general Kneser
hypergraphs via Tucker's lemma \cite{Tucker46} and Fan's lemma \cite{Ky
Fan52}\textit{\ }has been extensively studied in the literature, see \cite%
{Alishahi and Hajiabolhassan15,Alishahi17,Alishahi and
Hajiabolhassan17,Chang Liu and Zhu 13,Pengan11,Pengan15,Hajiabolhassan and
Meunier16,Liu and Zhu16,Matousek04,Meunier11,Meunier14,Sani and
Alishahi17,Simonyi and Tardos06,Simonyi and Tardos07,Simonyi and
Tardif13,Ziegler02,Ziegler06}.$\newline
$The following lemma was proposed by Meunier \cite{Meunier11}. It is a
variant of the $Z_{p}$-Tucker lemma by Ziegler \cite{Ziegler02,Ziegler06}.

\begin{lemma}[$Z_{p}$-Tucker lemma]
\label{Zp-Tucker}Let $p$ be a prime, $n,m\geq 1$, $\alpha \leq m$, and let%
\begin{equation*}
\begin{array}{cccc}
\lambda : & (Z_{p}\cup \{0\})^{n}\backslash \{(0,\ldots ,0)\} & 
\longrightarrow & Z_{p}\times \lbrack m] \\ 
& X & \mapsto & (\lambda _{1}(X),\lambda _{2}(X))%
\end{array}%
\end{equation*}%
be a $Z_{p}$-equivariant map satisfying the following properties:\newline
(i) for all $X^{(1)}\subseteq X^{(2)}\in (Z_{p}\cup \{0\})^{n}\backslash
\{(0,\ldots ,0)\}$, if $\lambda _{2}\left( X^{(1)}\right) =\lambda
_{2}\left( X^{(2)}\right) \leq \alpha $, then $\lambda _{1}\left(
X^{(1)}\right) =\lambda _{1}\left( X^{(2)}\right) $, and\newline
(ii) for all $X^{(1)}\subseteq X^{(2)}\subseteq \ldots \subseteq X^{(p)}\in
(Z_{p}\cup \{0\})^{n}\backslash \{(0,\ldots ,0)\}$, if $\lambda _{2}\left(
X^{(1)}\right) =\lambda _{2}\left( X^{(2)}\right) =\ldots =\lambda
_{2}\left( X^{(p)}\right) \geq \alpha +1$, then $\lambda _{1}\left(
X^{(i)}\right) $ are not pairwise distinct for $i=1,\ldots ,p$. \newline
Then $\alpha +(m-\alpha )(p-1)\geq n$.
\end{lemma}

Let $\mathcal{H}=(V(\mathcal{H}),E(\mathcal{H}))$ be a hypergraph and $r$ be
an integer, where $r\geq 2$. The \textit{general Kneser hypergraph} $KG^{r}(%
\mathcal{H})$ is a hypergraph with the vertex set $E(\mathcal{H}$) and the
edge set 
\begin{equation*}
E(KG^{r}(\mathcal{H}))=\{\{e_{1},...,e_{r}\}:e_{i}\in E(\mathcal{H})\text{
and }e_{i}\cap e_{j}=\text{\O\ for each }i\neq j\in \lbrack r]\}\text{.}
\end{equation*}

In 2011, Meunier \cite{Meunier11} investigated the chromatic number of
almost $2$-stable $r$-uniform\ Kneser hypergraphs. Dol'nikov-K\v{r}\'{\i}%
\v{z} Theorem \ref{t00006} was improved from the work of Alishahi and
Hajiabolhassan \cite{Alishahi and Hajiabolhassan15} and their interesting
notion of \textit{alternation number} for general Kneser hypergraphs.
Alishahi and Hajiabolhassan \cite{Alishahi and Hajiabolhassan15} generalized
the proof techniques of Meunier \cite{Meunier11}. Frick \cite{Frick16}
investigated $\chi \left( KG^{r}(n,k)_{s\text{-}stab}\right) $ for the case $%
r>s$. Recently, he shows that Conjecture \ref{c0001} is true for $r>6s-6$ a
prime power. Frick \cite{Frick17} makes significant progress via the
topological Tverberg theorem. The proof techniques of Frick do not apply to
the case $r\leq s$.

Let $p$ be a prime number. With the help of $Z_{p}$-Tucker's Lemma \ref%
{Zp-Tucker}, we obtain the following topological lower bound of the
chromatic number of the almost $s$-stable $p$-uniform\ Kneser hypergraph $%
KG^{p}(n,k)_{\widetilde{s\text{-}stab}}$. Our method is different from those
of \cite{Alishahi and Hajiabolhassan15}, \cite{Frick16,Frick17}, and \cite%
{Meunier11}.

\begin{theorem}
\label{t00001}Let $p$ be a prime number and $n,k,s$ be positive integers
such that $s\geq 2$ and $n\geq (p+s-2)(k-1)+\max \left(
\{p,s\}\right) $. Then $\chi \left( KG^{p}(n,k)_{\widetilde{s\text{-}stab}%
}\right) \geq \left\lceil \frac{n-(p+s-2)(k-1)}{p-1}\right\rceil $.
\end{theorem}

\begin{proof}
Assume that $KG^{p}(n,k)_{\widetilde{s\text{-}stab}}$ is properly colored
with $t$ colors. For $S\in \binom{\left[ n\right] }{k}_{\widetilde{s\text{-}%
stab}}$, we denote by $c(S)$ its color. We know that if $A$ is a nonempty
subset of $[n]$, then $A$ must contain an almost $s$-stable $1$-subset of $%
[n]$. In our approach, we need to introduce three functions. It should be
emphasized that we shall use two functions $\mathcal{I}(-)$ and $\mathcal{C}%
(-)$ several times during the proof. Let $X=(x_{1},x_{2},\ldots ,x_{n})\in
(Z_{p}\cup \{0\})^{n}\backslash \{0\}^{n}$. We can write alternatively $%
X=(X_{1},X_{2},\ldots ,X_{p})$. Define

$\mathcal{I}(X):=\max (\{q\in \lbrack k]:T$ is an almost $s$-stable $q$%
-subset of $[n]$ and $T\subseteq X_{j}$ for some\ $j\in \lbrack p]\})$.

Clearly, $\mathcal{I}(X)$ is well-defined. Then we define

$\mathcal{C}(X):=\{j\in \lbrack p]:X_{j}$ contains an almost $s$-stable $%
\mathcal{I}(X)$-subset of $[n]\}$, and

$\mathcal{M}(X):=\max (\{\max \left( T\right) \in \lbrack n]:T$ is an almost 
$s$-stable $\mathcal{I}(X)$-subset of $[n]$ and $T\subseteq X_{j}$ for some $%
j\in \lbrack p]\})$.

By definition of almost $s$-stable set, we know that if $X_{j}$ contains an
almost $s$-stable $\mathcal{I}(X)$-subset $A$ for some $j\in \lbrack p]$,
then $X_{j}$ must contain an almost $s$-stable $\mathcal{I}(X)$-subset $B$
with $\max \left( B\right) =\max \left( X_{j}\right) $. Hence, we can derive
that

$\mathcal{M}(X):=\max (\{\max \left( X_{j}\right) \in \lbrack n]:$ $j\in 
\mathcal{C}(X)\})$.

One can easily check that if $\left\vert \mathcal{C}(X)\right\vert =p$, then 
$\mathcal{M}(X)=\max \left( X\right) $.

Set $\alpha =(s-1)(k-1)$. Now, define the function%
\begin{equation*}
\begin{array}{cccc}
\lambda : & (Z_{p}\cup \{0\})^{n}\backslash \{(0,\ldots ,0)\} & 
\longrightarrow & Z_{p}\times \lbrack m]%
\end{array}%
\end{equation*}%
with $m=\alpha +(k-1)+t$. We choose a total ordering $\preceq $ on the
subsets of $[n]$. This ordering is only used to get a clean definition of $%
\lambda $.

\textbf{Case I.} $\mathcal{I}(X)\leq k-1$.

If $\left\vert \mathcal{C}(X)\right\vert =p$, then we know that $\mathcal{M}%
(X)=\max \left( X\right) $. Let $j$ be the index of $X_{j}$ with $\max
\left( X_{j}\right) =\max (X)$.

Define $\lambda \left( X\right) $ as%
\begin{equation*}
\left( \omega ^{j},\left( \mathcal{I}(X)-1\right) (s-1)+\left( \max
(X)-(s-1)\left\lfloor \frac{\max (X)}{s-1}\right\rfloor \right) +1\right) 
\text{.}
\end{equation*}

Note that $\max (X)-(s-1)\left\lfloor \frac{\max (X)}{s-1}\right\rfloor $ is
the remainder of $\max (X)$ divided by $s-1$. One can easily check that $%
\lambda _{2}\left( X\right) \in \{1,2,\ldots ,\alpha \}$.

If $\left\vert \mathcal{C}(X)\right\vert <p$, let $j$ be the index of $X_{j}$
with $\max \left( X_{j}\right) =\mathcal{M}(X)$ and then $\lambda \left(
X\right) $ is defined to be $\left( \omega ^{j},\alpha +\mathcal{I}%
(X)\right) $. One can easily check that $\lambda _{2}\left( X\right) \in
\{\alpha +1,\ldots ,\alpha +(k-1)\}$.\newline
\ \ \ \ \ \ \ \ \ \ \ \ \textbf{Case II.} $\mathcal{I}(X)=k$. By definition
of $\mathcal{I}(X)$, at least one of the $X_{j}$'s with $j\in \lbrack p]$
contains an almost $s$-stable $k$-subsets of $[n]$. Choose $j\in \lbrack p]$
such that there is $S\subseteq X_{j}$ with $S\in \binom{\left[ n\right] }{k}%
_{\widetilde{s\text{-}stab}}$. In case several $S$ are possible, choose the
maximal one according to the total ordering $\preceq $. Let $j$ be such that 
$S\subseteq X_{j}$ and define%
\begin{equation*}
\lambda (X):=\left( \omega ^{j},c(S)+s(k-1)\right) \text{.}
\end{equation*}

One can easily check that $\lambda _{2}\left( X\right) \in \{\alpha
+k,\ldots ,m\}$. Clearly, $\lambda $ is an $Z_{p}$-equivariant map from $%
(Z_{p}\cup \{0\})^{n}\backslash \{(0,\ldots ,0)\}$ to $Z_{p}\times \lbrack
m] $.

Let $X^{(1)}\subseteq X^{(2)}\in (Z_{p}\cup \{0\})^{n}\backslash \{(0,\ldots
,0)\}$. Obviously, $\max (X^{(1)})\leq \max (X^{(2)})$ and $\mathcal{I}%
(X^{(1)})\leq \mathcal{I}(X^{(2)})$. If $\lambda _{2}\left( X^{(1)}\right)
=\lambda _{2}\left( X^{(2)}\right) \leq \alpha $, then $\left\vert \mathcal{C%
}(X^{(1)})\right\vert =\left\vert \mathcal{C}(X^{(2)})\right\vert =p$, $%
\mathcal{I}(X^{(1)})\leq \mathcal{I}(X^{(2)})\leq k-1$, and $\left( \mathcal{%
I}(X^{(1)})-1\right) (s-1)+\max (X^{(1)})-(s-1)\left\lfloor \frac{\max
(X^{(1)})}{s-1}\right\rfloor +1=\left( \mathcal{I}(X^{(2)})-1\right)
(s-1)+\max (X^{(2)})-(s-1)\left\lfloor \frac{\max (X^{(2)})}{s-1}%
\right\rfloor +1$. It implies that $0\leq \left( \mathcal{I}(X^{(2)})-%
\mathcal{I}(X^{(1)})\right) (s-1)=\left( \max (X^{(1)})-(s-1)\left\lfloor 
\frac{\max (X^{(1)})}{s-1}\right\rfloor \right) -\left( \max
(X^{(2)})-(s-1)\left\lfloor \frac{\max (X^{(2)})}{s-1}\right\rfloor \right)
\leq s-2$. Hence, we have $\mathcal{I}(X^{(1)})=\mathcal{I}(X^{(2)})$ and $%
\max (X^{(1)})-(s-1)\left\lfloor \frac{\max (X^{(1)})}{s-1}\right\rfloor
=\max (X^{(2)})-(s-1)\left\lfloor \frac{\max (X^{(2)})}{s-1}\right\rfloor $.
That is, 
\begin{equation}
\mathcal{I}(X^{(1)})=\mathcal{I}(X^{(2)})\text{ and }\max (X^{(1)})\equiv
\max (X^{(2)})\text{ }\left( mod\text{ }s-1\right) \text{.}  \label{r0001}
\end{equation}%
Now we show that $\lambda _{1}\left( X^{(1)}\right) =\lambda _{1}\left(
X^{(2)}\right) $. Assume that $\lambda _{1}\left( X^{(1)}\right) =\omega
^{i}\neq \omega ^{j}=\lambda _{1}\left( X^{(2)}\right) $ for $i\neq j\in
\lbrack p]$. It means that $\max \left( X_{i}^{(1)}\right) =\max
(X^{(1)})<\max (X^{(2)})=\max \left( X_{j}^{(2)}\right) $. Since $\left\vert 
\mathcal{C}(X^{(1)})\right\vert =p$, there exist two disjoint almost $s$%
-stable $\mathcal{I}(X^{(1)})$-subsets $A$ and $B$ such that $A\subseteq
X_{j}^{(1)}$, $B\subseteq X_{i}^{(1)}$, and $\max (A)<\max (B)=\max
(X^{(1)})<\max (X^{(2)})=\max \left( X_{j}^{(2)}\right) $. From (\ref{r0001}%
), we obtain that $\max \left( X_{j}^{(2)}\right) -\max (A)=(\max
(X^{(2)})-\max (X^{(1)}))+(\max (B)-\max (A))\geq (s-1)+1=s$. Since $%
X^{(1)}\subseteq X^{(2)}$, $A\subseteq X_{j}^{(1)}\subseteq X_{j}^{(2)}$. It
means that there is an almost $s$-stable $\left( \mathcal{I}%
(X^{(1)})+1\right) $-subset $F=A\cup \{\max (X_{j}^{(2)})\}$ such that $%
F\subseteq X_{j}^{(2)}$. It follows that $\mathcal{I}(X^{(2)})\geq \mathcal{I%
}(X^{(1)})+1>\mathcal{I}(X^{(1)})$. It contradicts to (\ref{r0001}). So we
are done.

Let $X^{(1)}\subseteq X^{(2)}\subseteq \cdots \subseteq X^{(p)}\in
(Z_{p}\cup \{0\})^{n}\backslash \{(0,\ldots ,0)\}$. If $\alpha +1\leq
\lambda _{2}\left( X^{(1)}\right) =\lambda _{2}\left( X^{(2)}\right) =\cdots
=\lambda _{2}\left( X^{(p)}\right) \leq \alpha +(k-1)$, then $\mathcal{I}%
(X^{(1)})=\mathcal{I}(X^{(2)})=\cdots =\mathcal{I}(X^{(p)})\leq k-1$.
Moreover, we have that for each $i\in \lbrack p]$, $\left\vert \mathcal{C}%
(X^{(i)})\right\vert <p$ and, there is almost $s$-stable $\mathcal{I}%
(X^{(p)})$-subset $T_{i}$ and $j_{i}\in \lbrack p]$ such that we have $%
T_{i}\subseteq X_{j_{i}}^{(i)}$. Since $X^{(1)}\subseteq X^{(2)}\subseteq
\cdots \subseteq X^{(p)}$, $T_{i}\subseteq X_{j_{i}}^{(i)}\subseteq
X_{j_{i}}^{(p)}$ for each $i\in \lbrack p]$. If all $\lambda _{1}\left(
X^{(i)}\right) $ would be distinct, then it would mean that all $j_{i}$
would be distinct, which implies that $\left\vert \mathcal{C}%
(X^{(p)})\right\vert =p$. It is a contradiction.

If $\lambda _{2}\left( X^{(1)}\right) =\lambda _{2}\left( X^{(2)}\right)
=\cdots =\lambda _{2}\left( X^{(p)}\right) \geq \alpha +k$, then for each $%
i\in \lbrack p]$, there is $S_{i}\in \binom{\left[ n\right] }{k}_{\widetilde{%
s\text{-}stab}}$ and $j_{i}\in \lbrack p]$ such that we have $S_{i}\subseteq
X_{j_{i}}^{(i)}$ and $\lambda _{2}\left( X^{(i)}\right) =c(S_{i})+\alpha
+(k-1)$. If all $\lambda _{1}\left( X^{(i)}\right) $ would be distinct, then
it would mean that all $j_{i}$ would be distinct, which implies that the $%
S_{i}$ would be disjoint but colored with the same color, which is
impossible since $c$ is a proper coloring.

Therefore, we can thus apply the $Z_{p}$-Tucker Lemma \ref{Zp-Tucker} and
conclude that $n\leq \alpha +(m-\alpha )(p-1)=(s-1)(k-1)+[(k-1)+t](p-1)$,
that is%
\begin{equation*}
t\geq \left\lceil \frac{n-(p+s-2)(k-1)}{p-1}\right\rceil \text{.}
\end{equation*}
\end{proof}

Combining Lemma \ref{l00001} and Theorem \ref{t00001} for the case $s=2$, we
obtain the following result proposed by Meunier \cite{Meunier11} via a
different approach. As an approach to Conjecture \ref{c000}, Meunier \cite%
{Meunier11} settled Conjecture \ref{c0004} for $s=2$.

\begin{corollary}
Let $p$ be a prime number and $n,k$ be positive integers such that $n\geq pk$%
. Then $\chi \left( KG^{p}(n,k)_{\widetilde{2\text{-}stab}}\right)
=\left\lceil \frac{n-p(k-1)}{p-1}\right\rceil $.
\end{corollary}

By Theorem \ref{t00001}, we also obtain a topological lower bound of the $s$%
-stable $p$-uniform\ Kneser hypergraph $KG^{p}(n,k)_{s\text{-}stab}$ as
follows.

\begin{theorem}
\label{t00002}Let $p$ be a prime number and $n,k,s$ be positive integers
such that $p\geq 2$, $s\geq 2$ and $n\geq (p+s-2)(k-1)+(s-1)+\max \left(
\{p,s\}\right) $. Then $\chi \left( KG^{p}(n,k)_{s\text{-}stab}\right) \geq
\left\lceil \frac{n-(p+s-2)(k-1)-(s-1)}{p-1}\right\rceil $.
\end{theorem}

\begin{proof}
Let $\overline{n}=n-s+1\geq (p+s-2)(k-1)+\max \left( \{p,s\}\right) \geq sk$%
. Assume that $KG^{p}(n,k)_{s\text{-}stab}$ is properly colored with $t$
colors. We know that if $S$ is an almost $s$-stable $k$-subset of $[n-s+1]$,
then $S$ is an $s$-stable $k$-subset of $[n]$. That is, $\binom{\left[ 
\overline{n}\right] }{k}_{\widetilde{s\text{-}stab}}\subseteq \binom{\left[ n%
\right] }{k}_{s\text{-}stab}$. It means that $KG^{p}(\overline{n},k)_{%
\widetilde{s\text{-}stab}}$ is $t$-colorable, and then $\chi \left( KG^{p}(%
\overline{n},k)_{\widetilde{s\text{-}stab}}\right) \leq t$. From Theorem \ref%
{t00001}, we have 
\begin{equation*}
t\geq \left\lceil \frac{\overline{n}-(p+s-2)(k-1)}{p-1}\right\rceil
=\left\lceil \frac{n-(p+s-2)(k-1)-(s-1)}{p-1}\right\rceil \text{.}
\end{equation*}
\end{proof}

By Theorem \ref{t00002} for the case $s=2$, we obtain the following result
proposed by Alishahi and Hajiabolhassan \cite{Alishahi and Hajiabolhassan15}
via a different approach.

\begin{corollary}
Let $p$ be a prime number and $n,k$ be positive integers such that $n\geq
pk+1$. Then $\chi \left( KG^{p}(n,k)_{2\text{-}stab}\right) \geq \left\lceil 
\frac{n-p(k-1)-1}{p-1}\right\rceil $.
\end{corollary}

Alishahi and Hajiabolhassan \cite{Alishahi and Hajiabolhassan15} confirmed
the chromatic number of $2$-stable $r$-uniform\ Kneser hypergraph $%
KG^{r}(n,k)_{\widetilde{2\text{-}stab}}$ unless $r$ is odd and $n\equiv k$ $%
(mod$ $r-1)$. Recently, Frick \cite{Frick17} completely confirm Conjecture %
\ref{c0001} for $s=2$.

Let $p$ be a prime number. Our results specialize to a substantial
improvement of the Dol'nikov-K\v{r}\'{\i}\v{z} lower bounds for $\chi \left(
KG^{p}(n,k)_{s\text{-}stab}\right) $ and $\chi \left( KG^{p}(n,k)_{%
\widetilde{s\text{-}stab}}\right) $ as well.

\subsection{Coloring results about $KG^{2}(n,k)_{\widetilde{s\text{-}stab}}$}

This subsection deals with some coloring results about the almost $s$-stable
Kneser graphs $KG^{2}(n,k)_{\widetilde{s\text{-}stab}}$. As a special case $%
r=2$ of Conjecture \ref{c0001}, Meunier \cite{Meunier11} proposed the
following conjecture:

Let $n$, $k$, $s$ be positive integers such that $k\geq 2$, $s\geq 2$ and $%
n\geq sk$. Then 
\begin{equation*}
\chi \left( KG^{2}(n,k)_{s\text{-}stab}\right) =n-s(k-1)\text{.}
\end{equation*}%
Meunier \cite{Meunier11} showed that $\chi \left( KG^{2}(sk+1,k)_{s\text{-}%
stab}\right) =s+1$. Jonsson \cite{Jonsson12} confirmed that $\chi \left(
KG^{r}(n,k)_{s\text{-}stab}\right) =n-s(k-1)$ for $s\geq 4$, provided $n$ is
sufficiently large in terms of $s$ and $k$. For the case $r=2$ and $s$ even,
Chen \cite{Pengan15} proved that Conjecture \ref{c0001} is true. It follows
that $\chi \left( KG^{2}\left( n,k\right) _{\widetilde{s\text{-}stab}%
}\right) =n-s(k-1)$ for all even $s$ and for $s\geq 4$ and $n$ sufficiently
large. Besides the case $s=2$, the case $s=3$ is completely open.

The proof of Theorem \ref{t00001} makes use of the $Z_{p}$-Tucker Lemma \ref%
{Zp-Tucker}. We provide a different approach to investigate the topological
lower bound of the chromatic number of the almost $s$-stable $r$-uniform\
Kneser hypergraphs. We consider the case $r=2$. Combining Lemma \ref{l00001}
and Theorem \ref{t00001}, we completely confirm the chromatic number of the
usual almost $s$-stable Kneser graphs is equal to $n-s(k-1)$ for all $s\geq
2 $.

\begin{theorem}
\label{t003}Let $n$, $k$, $s$ be positive integers such that $k\geq 2$, $%
s\geq 2$ and $n\geq sk$. Then%
\begin{equation*}
\chi \left( KG^{2}\left( n,k\right) _{\widetilde{s\text{-}stab}}\right)
=n-s(k-1)\text{.}
\end{equation*}
\end{theorem}

Recall the definition of $(Z_{p}\cup \{0\})^{n}$ for $p=2$. We write $%
\{+,-,0\}^{n}$ for the set of all \textit{signed subsets} of $[n]$, the
family of all pairs $(X^{+},X^{-})$ of disjoint subsets of $[n]$. Such
subsets can alternatively be encoded by \textit{sign vectors} $X\in
\{+,-,0\}^{n}$, where $X_{i}=+$ denotes that $i\in $ $X^{+}$, while $X_{j}=-$
\ means that $j\in $ $X^{-}$. The \textit{positive part} of $X$ is $X^{+}:=$ 
$\{i\in \lbrack n]:X_{i}=+\}$, and analogously for the \textit{negative part}
$X^{-}$. For every signed subset $(X^{+},X^{-})$ in $\{+,-,0\}^{n}$, the
signed subset $(X^{-},X^{+})$ can be encoded by sign vector $-X$, that is, $%
\left( -X\right) ^{+}=X^{-}$ and $\left( -X\right) ^{-}=X^{+}$. For example, 
$-(0+-+)=(0-+-)$. We write $\left\vert X\right\vert $ for the number of
non-zero signs in $X$, that is, $\left\vert X\right\vert =\left\vert
X^{+}\right\vert +\left\vert X^{-}\right\vert $. Let $X\in
\{+,-,0\}^{n}\backslash \{0\}^{n}$. We write $\max \left( X\right) $ for the
maximal element of $X^{+}\cup X^{-}$, that is, $\max \left( X\right) =\max
\left( X^{+}\cup X^{-}\right) $.

In the following, we shall switch freely between the different notations for
signed sets. For sign vectors, we use the usual partial order from oriented
matroid theory, which is defined componentwise with $0\leq +$ and $0\leq -$.
Thus $X\leq Y$, that is $(X^{+},X^{-})\leq (Y^{+},Y^{-})$, holds if and only
if $X^{+}\subseteq Y^{+}$ and $X^{-}\subseteq Y^{-}$. For positive integers $%
m$ and $n$, and an antipodal subset $\mathcal{S}$ $\subseteq \{+,-,0\}^{n}$ $%
\backslash \{0\}^{n}$, a labeling map $\lambda :\mathcal{S}$ $%
\longrightarrow \{\pm 1,\pm 2,\ldots ,\pm m\}$ is called \textit{antipodal}
if $\lambda \left( -X\right) =-\lambda \left( X\right) $ for all $X$.

Let $n$ be a positive integer and let $sd([-1,+1]^{n})$ be the \textit{%
barycentric subdivision} of the $n$-cube $[-1,+1]^{n}$. We denote $\partial
\left( sd([-1,+1]^{n})\right) $ to be the boundary of $sd([-1,+1]^{n})$. In
the rest of this paper, the vertex set of $\partial \left(
sd([-1,+1]^{n})\right) $ can be identified with $\{+,-,0\}^{n}\backslash
\{0\}^{n}$. Every set with $n$ vertices in $\{+,-,0\}^{n}\backslash
\{0\}^{n} $ form an $\left( n-1\right) $-simplex if they can be arranged in
a sequence $X_{1}\leq X_{2}\leq \cdots \leq X_{n}$ satisfying $\left\vert
X_{j}\right\vert =j$ for $j=1,2,\ldots ,n$. Fan $\cite{Ky Fan52}$ proposed a
combinatorial formula on the barycentric subdivision of the octahedral
subdivision of $n$-sphere $S^{n}$. The following lemma corresponds to Fan's
combinatorial lemma $\cite{Ky Fan52}$ applied to $\partial \left(
sd([-1,+1]^{n})\right) $.

\begin{lemma}[Octahedral Fan's lemma]
\label{l111}Let $m$, $n$ be positive integers. Suppose $\lambda
:\{+,-,0\}^{n}\backslash \{0\}^{n}\longrightarrow \{\pm 1,\pm 2,\ldots ,\pm
m\}$\ satisfies (i) $\lambda $ is antipodal and (ii) $X\leq Y$ implies $%
\lambda \left( X\right) \neq -\lambda \left( Y\right) $ for all $X,Y$. Then
there are $n$ signed sets $X_{1}\leq X_{2}\leq \cdots \leq X_{n}$ in $%
\{+,-,0\}^{n}\backslash \{0\}^{n}$ such that $\{\lambda (X_{1}),\lambda
(X_{2}),\ldots ,\lambda (X_{n})\}=\{+a_{1},-a_{2},\ldots ,(-1)^{n-1}a_{n}\}$%
, where $1\leq a_{1}<a_{2}<\cdots <a_{n}\leq m$. In particular, $m\geq n$.
\end{lemma}

We say that a graph is \textit{completely multicolored} in a coloring if all
its vertices receive different colors. The existence of large colorful
bipartite subgraphs in a properly colored graph has been extensively studied
in the literature, see \cite{Alishahi and
Hajiabolhassan15,Alishahi17,Alishahi and Hajiabolhassan17,Chang Liu and Zhu
13,Pengan11,Liu and Zhu16,Meunier14,Simonyi and Tardos06,Simonyi and
Tardos07,Simonyi and Tardif13}. Simonyi and Tardos in 2007 \cite{Simonyi and
Tardos07} improved Dol'nikov's theorem. The special case for Kneser graphs
is due to Ky Fan \cite{Ky Fan82}.

\begin{theorem}[Simonyi-Tardos theorem]
\label{t00005} Let $\mathcal{H}$ be a hypergraph and assume that \O\ is not
an edge of $\mathcal{H}$. Let $r=$ cd$^{2}(\mathcal{H})$. Then any proper
coloring of $KG^{2}(\mathcal{H})$ with colors $1,\ldots ,t$ ($t$ arbitrary)
must contain a completely multicolored complete bipartite graph $%
K_{\left\lceil \frac{r}{2}\right\rceil ,\left\lfloor \frac{r}{2}%
\right\rfloor }$ such that the $r$ different colors occur alternating on the
two parts of the bipartite graph with respect to their natural order.
\end{theorem}

In 2014, Meunier \cite{Meunier14} found the first colorful type result for
uniform hypergraphs to generalize Simonyi and Tardos's work \cite{Simonyi
and Tardos07}. Follow the proof technique of Theorem \ref{t00001}, we
propose the following strengthened result of the Simonyi-Tardos Theorem \ref%
{t00005} about the case $\mathcal{H=}\binom{\left[ n\right] }{k}_{\widetilde{%
s\text{-}stab}}$ via the Octahedral Fan's Lemma \ref{l111}.

\begin{theorem}
\label{t00003}Let $n$, $k$, $s$, and $t$ be positive integers such that $%
k\geq 2$, $s\geq 2$, $n\geq sk$, and $t=n-s(k-1)$. Then any proper coloring
of $KG^{2}\left( n,k\right) _{\widetilde{s\text{-}stab}}$ with $\{1,2,\ldots
,m\}$ colors ($m$ arbitrary) must contain a completely multicolored complete
bipartite subgraph $K_{\left\lceil \frac{t}{2}\right\rceil \left\lfloor 
\frac{t}{2}\right\rfloor }$ such that the $t$ different colors occur
alternating on the two sides of the bipartite graph with respect to their
natural order.
\end{theorem}

\begin{proof}
Assume that $KG(n,k)_{\widetilde{s\text{-}stab}}$ is properly colored with $%
m $ colors. For $S\in \binom{\left[ n\right] }{k}_{\widetilde{s\text{-}stab}%
} $, we denote by $c(S)$ its color.

Let $X=(X^{+},X^{-})\in \{+,-,0\}^{n}\backslash \{0\}^{n}$. Set

$\mathcal{I}(X):=\max (\{q\in \lbrack k]:T$ is an almost $s$-stable subset
of size $q$, and $T\subseteq X^{+}$ or $T\subseteq X^{-}\})$.

We define the function%
\begin{equation*}
\lambda :\{+,-,0\}^{n}\backslash \{0\}^{n}\longrightarrow \{\pm 1,\pm
2,\ldots ,\pm (s(k-1)+m)\}
\end{equation*}%
as follows.

\textbf{Case I.} $\mathcal{I}(X)\leq k-1$.

If there exist two disjoint almost $s$-stable $\mathcal{I}(X)$-subsets $S,T$
such that $S\subseteq X^{+}$ and $T\subseteq X^{-}$, then we define $\lambda
(X)$ as\newline
\begin{equation*}
\pm \left( \left( \mathcal{I}(X)-1\right) (s-1)+\left( \max
(X)-(s-1)\left\lfloor \frac{\max (X)}{s-1}\right\rfloor \right) +1\right) 
\text{,}
\end{equation*}%
where the sign indicates which of $\max \left( X^{+}\right) $ or $\max
\left( X^{-}\right) $ equals $max(X)$. Note that $\max (X)-(s-1)\left\lfloor 
\frac{\max (X)}{s-1}\right\rfloor $ is the remainder of $\max (X)$ divided
by $s-1$. Thus we obtain a value $\lambda (X)$ in the set $\{\pm 1,\pm
2,\ldots ,\pm (s-1)(k-1)\}$.

Otherwise, define $\lambda (X)$ as 
\begin{equation*}
\pm (\mathcal{I}(X)+(s-1)(k-1))\text{,}
\end{equation*}%
where the sign indicates which of $X^{-}$ or $X^{+}$ can contain an almost $%
s $-stable of size $\mathcal{I}(X)$. Thus we obtain a value $\lambda (X)$ in
the set $\{\pm ((s-1)(k-1)+1),\ldots ,\pm s(k-1)\}$.

\textbf{Case II.} $\mathcal{I}(X)=k$. By definition of $I(X)$, at least one
of $X^{+}$ and $X^{-}$ contains an almost $s$-stable $k$-subset. Among all
almost $s$-stable $k$-subsets included in $X^{+}$ and $X^{-}$, select the
one having the largest color. Call it $S$. Then define 
\begin{equation*}
\lambda (X)=\pm (c(S)+s(k-1))\text{,}
\end{equation*}%
where the sign indicates which of $X^{-}$ or $X^{+}$ the subset $S$ has been
taken from. Thus we obtain a value $\lambda (X)$ in the set $\{\pm
(s(k-1)+1),\ldots ,\pm (s(k-1)+m)\}$.

Clearly, $\lambda $ is antipodal. Follows the similar scheme as the proof of
Theorem \ref{t00001}, one can easily check that $X\leq Y$ implies $\lambda
\left( X\right) \neq -\lambda \left( Y\right) $ for all $X,Y$. Applying
Octahedral Fan's Lemma \ref{l111}, there are $n$ signed sets $X_{1}\leq
X_{2}\leq \cdots \leq X_{n}$ in $\{+,-,0\}^{n}\backslash \{0\}^{n}$ such
that $\{\lambda (X_{1}),\lambda (X_{2}),\ldots ,\lambda
(X_{n})\}=\{+a_{1},-a_{2},\ldots ,(-1)^{n-1}a_{n}\}$ where $1\leq
a_{1}<a_{2}<\cdots <a_{n}\leq m$. Since $1\leq \left\vert X_{1}\right\vert
<\left\vert X_{2}\right\vert <\cdots <\left\vert X_{n}\right\vert \leq n$,
we have $\left\vert X_{i}\right\vert =i$ for $i=1,2,\ldots ,n$. From the
definition of $\lambda $, we know that $\max (\{\lambda (X_{1}),\left\vert
\lambda (X_{2})\right\vert ,\ldots ,\left\vert \lambda
(X_{s(k-1)})\right\vert \})<\left\vert \lambda (X_{s(k-1)+1})\right\vert
<\left\vert \lambda (X_{s(k-1)+2})\right\vert <\cdots <\left\vert \lambda
(X_{n})\right\vert $. Therefore, we have $\lambda (X_{i})=(-1)^{i-1}a_{i}$
for $i=s(k-1)+1$, $s(k-1)+2$, $\ldots $, $n$. Assume that $n$ is even. If $%
s(k-1)$ is even, then there exist $A_{1},A_{3},\ldots ,A_{n-s(k-1)-1}$ in $%
\binom{\left[ n\right] }{k}_{\widetilde{s\text{-}stab}}$ such that $%
A_{i}\subseteq X_{s(k-1)+i}^{+}\subseteq X_{n}^{+}$ and $%
c(A_{i})=a_{s(k-1)+i}$ for $i=1$, $3$, $\ldots $, $n-s(k-1)-1$. Also, there
exist $A_{2},A_{4},\ldots ,A_{n-s(k-1)}$ in $\binom{\left[ n\right] }{k}_{%
\widetilde{s\text{-}stab}}$ such that $A_{i}\subseteq
X_{s(k-1)+i}^{-}\subseteq X_{n}^{-}$ and $c(A_{i})=a_{s(k-1)+i}$ for $i=2$, $%
4$, $\ldots $, $n-s(k-1)$. Since $X_{n}^{+}\cap X_{n}^{-}=$ \O , the induced
subgraph on vertices%
\begin{equation*}
\{A_{1},A_{3},\ldots ,A_{n-s(k-1)-1}\}\cup \{A_{2},A_{4},\ldots
,A_{n-s(k-1)}\}
\end{equation*}%
is a complete bipartite graph which is the desired subgraph. If $s(k-1)$ is
odd, then there exist $A_{1},A_{3},\ldots ,A_{n-s(k-1)}$ in $\binom{\left[ n%
\right] }{k}_{\widetilde{s\text{-}stab}}$ such that $A_{i}\subseteq
X_{s(k-1)+i}^{-}\subseteq X_{n}^{-}$\ and $c(A_{i})=a_{s(k-1)+i}$ for $i=1$, 
$3$, $\ldots $, $n-s(k-1)$. Also, there exist $A_{2},A_{4},\ldots
,A_{n-s(k-1)-1}$ in $\binom{\left[ n\right] }{k}_{\widetilde{s\text{-}stab}}$
such that $A_{i}\subseteq X_{s(k-1)+i}^{+}\subseteq X_{n}^{+}$ and $%
c(A_{i})=a_{s(k-1)+i}$ for $i=2$, $4$, $\ldots $, $n-s(k-1)-1$. Since $%
X_{n}^{+}\cap X_{n}^{-}=$ \O , the induced subgraph on vertices%
\begin{equation*}
\{A_{1},A_{3},\ldots ,A_{n-s(k-1)}\}\cup \{A_{2},A_{4},\ldots
,A_{n-s(k-1)-1}\}
\end{equation*}%
is a complete bipartite graph which is the desired subgraph. Now, assume
that $n$ is odd. We can prove the statement with the same kind of proof
works when $n$ is even (omitted here) directly. This completes the proof.
\end{proof}

In a graph $G=(V,E)$, the \textit{closed neighborhood} of a vertex $u$,
denoted $N[u]$, is the set $\{u\}\cup \{v:uv\in E\}$. The \textit{local
chromatic number} of a graph $G=(V;E)$, denoted $\chi _{l}(G)$, is the
maximum number of colors appearing in the closed neighborhood of a vertex
minimized over all proper colorings:%
\begin{equation*}
\chi _{l}(G)=\min_{c}\max_{v\in V}\left\vert c(N[v])\right\vert \text{,}
\end{equation*}%
where the minimum is taken over all proper colorings $c$ of $G$. This number
has been defined in 1986 by Erd\H{o}s, F\"{u}redi, Hajnal, Komj\'{a}th, R%
\"{o}dl, and Seress \cite{Erdos86}. Meunier \cite{Meunier14}, by using his
colorful theorem, generalized the Simonyi-Tardos lower bound \cite{Simonyi
and Tardos07} for the local chromatic number of Kneser graphs to the local
chromatic number of Kneser hypergraphs. The local chromatic number of
uniform Kneser hypergraphs have been extensively studied in the literature,
see \cite{Alishahi and Hajiabolhassan15,Alishahi17,Alishahi and
Hajiabolhassan17,Meunier14,Simonyi and Tardos06,Simonyi and Tardos07}. We
investigate the lower bound for the local chromatic number of almost $s$%
-stable Kneser graphs. The following result is an immediate consequence of
Theorem \ref{t00003}.

\begin{theorem}
Let $n$, $k$, and $s$ be positive integers such that $k\geq 2$, $s\geq 2$,
and $n\geq sk$. Then%
\begin{equation*}
\chi _{l}(KG^{2}\left( n,k\right) _{\widetilde{s\text{-}stab}})\geq
\left\lceil \frac{n-s(k-1)}{2}\right\rceil +1\text{.}
\end{equation*}
\end{theorem}

Our lower bound improves Simonyi and Tardos \cite{Simonyi and Tardos07}, and
Meunier's \cite{Meunier14} lower bounds for the local chromatic number of $%
KG^{2}\left( n,k\right) _{\widetilde{s\text{-}stab}}$ as well.

\end{document}